\documentclass[12pt]{article}
\def\Dj{\hbox{D\kern-.73em\raise.30ex\hbox{-}
\raise-.30ex\hbox{}}}
\def\dj{\hbox{d\kern-.33em\raise.80ex\hbox{-}
\raise-.80ex\hbox{\kern-.40em}}}
\usepackage{epsfig}
\usepackage{amsmath,amsthm,amsfonts,amssymb,amscd,cite}
\allowdisplaybreaks

\newtheorem{theorem}{Theorem}[section]
\newtheorem{lemma}[theorem]{Lemma}

\newtheorem{conjecture}[theorem]{Conjecture}
\newtheorem{problem}[theorem]{Problem}

\begin{document}

\baselineskip=0.30in
\begin{center}
{\Large \bf On the spectral characterization of Kite graphs}

\vspace{6mm}

{\large \bf Sezer Sorgun $^a$\,,} {\large \bf Hatice Topcu
}$^{a}$\

\vspace{9mm}

\baselineskip=0.20in

 $^a${\it Department of Mathematics,\\
 Nevþehir Hac{\i} Bekta\c{s} Veli University, \\
 Nevþehir 50300, Turkey.\/} \\
 e-mail: {\tt srgnrzs@gmail.com, haticekamittopcu@gmail.com}\\[3mm]

\vspace{6mm}

(Received June 4, 2015)

\end{center}

\vspace{6mm}

\baselineskip=0.23in

\begin{abstract}
 The \textit{Kite graph}, denoted by $Kite_{p,q}$  is obtained by appending a complete graph $K_{p}$ to a pendant vertex of a path $P_{q}$. In this paper, firstly we show that no two non-isomorphic kite graphs are cospectral w.r.t adjacency matrix. Let $G$ be a graph which is cospectral with $Kite_{p,q}$  and the clique number of $G$ is denoted by $w(G)$. Then, it is shown that $w(G)\geq p-2q+1$. Also, we prove that $Kite_{p,2}$  graphs are determined by their adjacency spectrum.

\bigskip


 \end{abstract}

 \vspace{5mm}

 \baselineskip=0.30in

\section{Introduction}

All of the graphs considered here are simple and undirected. Let $G=(V,E)$ be a graph with vertex set $V(G)=\{v_{1},v_{2},\ldots,v_{n}\}$ and edge set $E(G)$.For a given graph $F$, if $G$ does not contain $F$ as a subgraph, then $G$ is called $F-free$. A complete subgraph of$G$ is called a \textit{clique of G}. \textit{The clique number of G} is the number of vertices in the largest clique of $G$ and it is denoted by $w(G)$. Let $A(G)$ be the \textit{(0,1)-adjacency matrix of G} and $d_{k}$ the degree of the vertex $v_{k}$. The polynomial $P_{G}(\lambda)=det(\lambda I-A(G))$ is the \textit{characteristic polynomial of G} where $I$ is the identity matrix. Eigenvalues of the matrix $A(G)$ are called \textit{adjacency eigenvalues}. Since $A(G)$ is real and symmetric matrix, adjacency eigenvalues are all real numbers and will be ordered as $\lambda_{1}\geq\lambda_{2}\geq\ldots\geq\lambda_{n}$. \textit{Adjacency spectrum of the graph G} consists the adjacency eigenvalues with their multiplicities. The largest eigenvalue of a graph is known as its \textit{spectral radius}.

Two graphs $G$ and $H$ are said to be \textit{cospectral} if they have same spectrum (i.e. same characteristic polynomial). A graph $G$ is \textit{determined by adjacency spectrum}, shortly \textit{DAS}, if every graph cospectral with $G$ is isomorphic to $G$. It has been conjectured by the first author in \cite{6} that almost all graphs are determined by their spectrum, \textit{DS} for short. But it is difficult to show that a given graph is \textit{DS}. Up to now, only few graphs are proved to be \textit{DS} \cite{2,3,4,5,6,7,9,10,11,12,13,15}. Recently, some papers have been appeared that focus on some special graphs (oftenly under some conditions) and prove that these special graphs are \textit{DS} or \textit{non-DS} \cite{2,3,4,7,9,10,11,12,13,15}. For a recent widely survey, one can see \cite{6}.

The \textit{Kite graph}, denoted by $Kite_{p,q}$, is obtained by appending a complete graph with $p$ vertices $K_{p}$ to a pendant vertex of a path graph with $q$ vertices $P_{q}$. If $q=1$, it is called \textit{short kite graph}.

In this paper, firstly we obtain the characteristic polynomial of kite graphs and show that no two non-isomorphic kite graphs are cospectral w.r.t adjacency matrix. Then for a given graph $G$ which is cospectral with $Kite_{p,q}$, the clique number of $G$ is $w(G)\geq p-2q+1$. Also we prove that $Kite_{p,2}$ graphs are \textit{DAS}  for all $p$.

\section{Preliminaries}

First, we give some lemmas that will be used in the next sections of this paper.

\begin{lemma} \cite{3} Let $x_{1}$ be a pendant vertex of a graph $G$ and $x_{2}$ be the vertex which is adjacent to $x_{1}$. Let $G_{1}$ be the induced subgraph obtained from $G$ by deleting the vertex $x_{1}$. If $x_{1}$ and $x_{2}$ are deleted, the induced subgraph $G_{2}$ is obtained. Then,

\[P_{A(G)}(\lambda )=\lambda P_{A(G_{1})}(\lambda )-P_{A(G_{2})}(\lambda )\]

\end{lemma}

\begin{lemma} \cite{5} For $nxn$ matrices $A$ and $B$, followings are equivalent :

\textbf{(i)} $A$ and $B$ are cospectral

\textbf{(ii) }$A$ and $B$ have the same characteristic polynomial

\textbf{(iii) }$tr(A^{i})=tr(B^{i})$ for $i=1,2,...,n$
\end{lemma}

\begin{lemma} \cite{5} For the adjacency matrix of a graph $G$, the following parameters can be deduced from the spectrum;

\textbf{(i)} the number of vertices

\textbf{(ii)} the number of edges

\textbf{(iii)} the number of closed walks of any fixed length.
\end{lemma}

Let $N_{G}(H)$ be the number of subgraphs of a graph $G$ which are isomorphic to $H$ and let $N_{G}(i)$ be the number of closed walks of length $i$ in $G$.

\begin{lemma} \cite{12} The number of closed walks of length 2 and 3 of a graph $G$ are given in the following, where m is number of edges of $G$.

\textbf{(i)} $N_{G}(2)=2m$ and $N_{G}(3)=6N_{G}(K_{3})$.
\end{lemma}

In the rest of the paper, we denote the number of subgraphs of a graph $G$ which are isomorphic to complete graph $K_{3}$ with $t(G)$.

\begin{theorem}\cite{1} For any integers $p\geq3$ and $q\geq1$, if we denote the spectral radius of $A(Kite_{p,q})$ with $\rho(Kite_{p,q})$ then

\[p-1+\frac{1}{p^{2}}+\frac{1}{p^{3}}<\rho (Kite_{p,q})<p-1+\frac{1}{4p}+\frac{1}{p^{2}-2p}\]

\end{theorem}

\begin{theorem}\cite{14} Let $G$ be a graph with n vertices, m edges and spectral radius $\mu$. If $G$ is $K_{r+1}-free$, then

\[\mu\leq\sqrt{2m(\frac{r-1}{r})}\]
\end{theorem}

\begin{theorem}\cite{4} Let $K_{n}^{m}$ denote the graph obtained by attaching m pendant edges to a vertex of complete graph $K_{n-m}$. The graph $K_{n}^{m}$ and its complement are determined by their adjacency spectrum.
\end{theorem}

\section{Characteristic Polynomials of Kite Graphs}

We use similar method with \cite{3} to obtain the general form of characteristic polynomials of $Kite_{p,q}$ graphs. Obviously, if we delete the vertex with one degree from short kite graph, the induced subgraph will be the complete graph $K_{p}$. Then, by deleting the vertex with one degree and its adjacent vertex, we obtain complete graph with $p-1$ vertices, $K_{p-1}$. From Lemma 2.1, we get

\begin{eqnarray*}
P_{A(Kite_{p,1})}(\lambda ) &=&\lambda P_{A(K_{p})}(\lambda
)-P_{A(K_{p-1})}(\lambda )  \notag \\
&=&\lambda (\lambda -p+1)(\lambda +1)^{p-1}-[(\lambda -p+2)(\lambda
+1)^{p-2}]  \notag \\
&=&(\lambda +1)^{p-2}[(\lambda ^{2}-\lambda p+\lambda )(\lambda +1)-\lambda
+p-2]  \notag \\
&=&(\lambda +1)^{p-2}[\lambda ^{3}-(p-2)\lambda ^{2}-\lambda p+p-2]
\end{eqnarray*}

Similarly, for $Kite_{p,2}$ induced subgraphs will be $Kite_{p,1}$ and $K_{p}$ respectively. By Lemma 2.1, we get

\begin{eqnarray*}
P_{A(Kite_{p,2})}(\lambda ) &=&\lambda P_{A(Kite_{p,1})}(\lambda
)-P_{A(K_{p})})(\lambda ) \\
&=&\lambda (\lambda P_{A(K_{p})}(\lambda )-P_{A(K_{p-1})}(\lambda
))-P_{A(K_{p})})(\lambda ) \\
&=&(\lambda ^{2}-1)P_{A(K_{p})}(\lambda )-\lambda P_{A(K_{p-1})}(\lambda )
\end{eqnarray*}%

By using these polynomials, let us calculate the characteristic polynomial of $Kite_{p,q}$ where $n=p+q$. Again, by Lemma 2.1  we have

\begin{eqnarray*}P_{A(Kite_{p,1})}(\lambda ) &=&\lambda P_{A(K_{p})}(\lambda)-P_{A(K_{p-1})}(\lambda )
\end{eqnarray*}

Coefficients of above equation are $b_{1}=-1$, $a_{1}=\lambda$. Simultaneously, we get

\begin{eqnarray*}P_{A(Kite_{p,2})}(\lambda ) &=&(\lambda^{2}-1) P_{A(K_{p})}(\lambda)-\lambda P_{A(K_{p-1})}(\lambda )
\end{eqnarray*}

and coefficients of above equation are $b_{2}=-a_{1}=-\lambda$, $a_{2}=\lambda a_{1}-1=\lambda ^{2}-1$. Then for $Kite_{p,3}$, we have

\begin{eqnarray*}P_{A(Kite_{p,3})}(\lambda ) &=&\lambda P_{A(Kite_{p,2})}(\lambda)-P_{A(Kite_{p,1})})(\lambda ) \\ &=&(\lambda (\lambda ^{2}-1)-\lambda) P_{A(K_{p})}(\lambda)-((\lambda^{2}-1)P_{A(K_{p-1})}(\lambda ))
\end{eqnarray*}

and coefficients of above equation are $b_{3}=-a_{2}=-(\lambda^{2}-1),a_{3}=\lambda a_{2}-a_{1}=\lambda (\lambda ^{2}-1)-\lambda $. In the following steps, for $n\geq 3$, $a_{n}=\lambda a_{n-1}-a_{n-2}$. From this difference equation, we get

\begin{equation*}
a_{n}=\sum_{k=0}^{n}(\frac{\lambda +\sqrt{\lambda ^{2}-4}}{2})^{k}(\frac{%
\lambda -\sqrt{\lambda ^{2}-4}}{2})^{n-k}
\end{equation*}

Now, let $\lambda =2cos\theta $ and $u=e^{i\theta }$. Then, we have

\begin{equation*}
a_{n}=\sum_{k=0}^{n}u^{2k-n}=\frac{u^{-n}(1-u^{2n+2})}{1-u^{2}}
\end{equation*}

and by calculation the characteristic polynomial of any kite graph, $Kite_{p,q}$, where $n=p+q$, is
\begin{eqnarray*}
P_{A(\emph{Kite}_{p,q})}(u+u^{-1})
&=&a_{n-p}P_{A(K_{p})}(u+u^{-1})-a_{n-p-1}P_{A(K_{p-1})}(u+u^{-1}) \\
&=&\frac{u^{-n+p}(1-u^{2n-2p+2})}{1-u^{2}}%
.((u+u^{-1}-p+1).(u+u^{-1}+1)^{p-1}) \\
&&-\frac{u^{-n+p+1}(1-u^{2n-2p+4})}{1-u^{2}}%
.((u+u^{-1}-p+2).(u+u^{-1}+1)^{p-2}) \\
&=&\frac{u^{-n+p}(1+u-u^{-1})^{p-2}}{1-u^{2}}%
[(2-p).(1+u^{-1}-u^{2n-2p+2}-u^{2n-2p+3}) \\
&&+(u^{-2}-u^{2n-2p+4})]\\
&=&\frac{u^{-q}(1+u-u^{-1})^{p-2}}{1-u^{2}}%
[(2-p).(1+u^{-1}-u^{2q+2}-u^{2q+3}) \\
&&+(u^{-2}-u^{2q+4})]\\
\end{eqnarray*}

\begin{theorem} No two non-isomorphic kite graphs have the same adjacency spectrum.
\end{theorem}

\begin{proof} Assume that there are two cospectral kite graphs with number of vertices respectively, $p_{1}+q_{1}$ and $p_{2}+q_{2}$. Since they are cospectral,
they must have same number of vertices and same characteristic polynomials. Hence, $n=p_{1}+q_{1}=p_{2}+q_{2}$ and we get

\begin{equation*}
P_{A(Kite_{p_{1},q_{1}})}(u+u^{-1})=P_{A(Kite_{p_{2},q_{2}})}(u+u^{-1})
\end{equation*}%
i.e.%
\begin{eqnarray*}
&&\frac{u^{-n+p_{1}}(1+u-u^{-1})^{p_{1}-2}}{1-u^{2}}%
[(2-p_{1}).(1+u^{-1}-u^{2n-2p_{1}+2}-u^{2n-2p_{1}+3}) \\
&&+(u^{-2}-u^{2n-2p_{1}+4})] \\
&=&\frac{u^{-n+p_{2}}(1+u-u^{-1})^{p_{2}-2}}{1-u^{2}}%
[(2-p_{2}).(1+u^{-1}-u^{2n-2p_{2}+2}-u^{2n-2p_{2}+3}) \\
&&+(u^{-2}-u^{2n-2p_{2}+4}])
\end{eqnarray*}%
i.e.%
\begin{eqnarray*}
&&u^{p_{1}}.(1+u-u^{-1})^{^{p_{1}}}.[(2-p_{1}).(1+u^{-1}-u^{2n-2p_{1}+2}-u^{2n-2p_{1}+3})
\\
&&+(u^{-2}-u^{2n-2p_{1}+4})] \\
&=&u^{p_{2}}.(1+u-u^{-1})^{^{p_{2}}}.[(2-p_{2}).(1+u^{-1}-u^{2n-2p_{2}+2}-u^{2n-2p_{2}+3})
\\
&&+(u^{-2}-u^{2n-2p_{2}+4})]
\end{eqnarray*}%

Let $p_{1}>p_{2}$. It follows that $n-p_{2}>n-p_{1}$. Then, we
have
\begin{eqnarray*}
&&u^{p_{1}-p_{2}}.(1+u-u^{-1})^{^{p_{1}-p_{2}}}%
\{[(2-p_{1}).(1+u^{-1}-u^{2n-2p_{1}+2}-u^{2n-2p_{1}+3}) \\
&&+(u^{-2}-u^{2n-2p_{1}+4})]-[(2-p_{2}).(1+u^{-1}-u^{2n-2p_{2}+2}-u^{2n-2p_{2}+3})
\\
&&+(u^{-2}-u^{2n-2p_{2}+4})]\}=0
\end{eqnarray*}
By using the fact that $u\neq 0$ and $1+u+u^{-1}\neq 0$, we get%
\begin{eqnarray*}
f(u)
&=&[(2-p_{1}).(1+u^{-1}-u^{2n-2p_{1}+2}-u^{2n-2p_{1}+3})+(u^{-2}-u^{2n-2p_{1}+4})]
\\
&&-[(2-p_{2}).(1+u^{-1}-u^{2n-2p_{2}+2}-u^{2n-2p_{2}+3})+(u^{-2}-u^{2n-2p_{2}+4})]
\\
&=&0
\end{eqnarray*}%
Since $f(u)=0$, the derivation of $(2n-2p_{2}+5)$th of $f$  equals to
zero again. Thus, we have
\begin{equation*}
\lbrack
(p_{1}-2)(2n-2p_{2}+4)!(u^{-2n+2p_{2}-6})]-[(p_{2}-2).(2n-2p_{2}+4)!(u^{-2n+2p_{2}-6})]=0
\end{equation*}%
i.e.%
\begin{equation*}
\lbrack (p_{1}-2)-(p_{2}-2)].(u^{-2n+2p_{2}-6})=0
\end{equation*}%
i.e.%
\begin{equation*}
p_{1}=p_{2}
\end{equation*}%
since $u\neq 0$. This is a contradiction with our assumption that $p_{1}>p_{2}$. For $p_{2}>p_{1}$, we get the similar contradiction. So $p_{1}$ must be equal to $p_{2}$. Hence $q_{1}=q_{2}$ and these graphs are isomorphic.
\end{proof}

\section{Spectral Determination of $Kite_{p,2}$ Graphs}

\begin{lemma} Let $G$ be a graph which is cospectral with $Kite_{p,q}$. Then we get $$w(G)\geq p-2q+1$$.
\end{lemma}

\begin{proof}

Since $G$ is cospectral with $Kite_{p,q}$, from Lemma 2.3, $G$ has the same number of vertices, same number of edges and same spectrum with $Kite_{p,q}$. So, if $G$ has $n$ vertices and $m$ edges, $n=p+q$ and $m=\left(                                       \begin{array}{c}
                                         p \\
                                         2 \\
                                       \end{array}
                                     \right)+q = \frac{p^{2}-p+2q}{2}$. Also, $\rho(G)=\rho(Kite_{p,q})$. From Theorem 2.6, we say that if $\mu>\sqrt{2m(\frac{r-1}{r})}$ then $G$ isn't $K_{r+1}-free$. It means that, $G$ contains $K_{r+1}$ as a subgraph. Now, we claim that for $r<p-2q$, $\sqrt{2m(\frac{r-1}{r})}<\rho(G)$. By Theorem 2.5, we've already known that $p-1+\frac{1}{p^{2}}+\frac{1}{p^{3}}<\rho(G)$. Hence, we need to show that, when $r<p-2q$, $\sqrt{2m(\frac{r-1}{r})}<p-1+\frac{1}{p^{2}}+\frac{1}{p^{3}}$. Indeed,

\begin{eqnarray*}
(\sqrt{2m(\frac{r-1}{r})})^{2}-(p-1+\frac{1}{p^{2}}+\frac{1}{p^{3}})^{2}
&=&(p^{2}-p+2q)(r-1)-r(p-1+\frac{1}{p^{2}}+\frac{1}{p^{3}})^{2} \\
&=&(p^{2}-p+2q)(r-1)- \\ & & (\frac{r(p^{2}+p^{3})}{p^{5}})(2(p-1)+\frac{(p^{2}+p^{3})}{p^{5}}) \\
&=&(pr-p^{2}+p+(2q-1)r-2q)-\\ & &(\frac{r(p^{2}+p^{3})}{p^{5}})(2(p-1)+\frac{(p^{2}+p^{3})}{p^{5}}) \\
\end{eqnarray*}

By the help of \textit{Mathematica}, for $r<p-2q$ we can see

 \[(pr-p^{2}+p+(2q-1)r-2q)-(\frac{r(p^{2}+p^{3})}{p^{5}})(2(p-1)+\frac{(p^{2}+p^{3})}{p^{5}})<0\]

i.e.

\[(\sqrt{2m(\frac{r-1}{r})})^{2}-(p-1+\frac{1}{p^{2}}+\frac{1}{p^{3}})^{2}<0\]

i.e.

\[(\sqrt{2m(\frac{r-1}{r})})^{2}<(p-1+\frac{1}{p^{2}}+\frac{1}{p^{3}})^{2}\]

Since $\sqrt{2m(\frac{r-1}{r})}>0$ and $p-1+\frac{1}{p^{2}}+\frac{1}{p^{3}}>0$, we get

\[\sqrt{2m(\frac{r-1}{r})}<p-1+\frac{1}{p^{2}}+\frac{1}{p^{3}}<\rho(G)\]

Thus, we proved our claim and so $G$ contains $K_{r+1}$ as a subgraph such that $r<p-2q$. Consequently, $w(G)\geq p-2q+1$.

\end{proof}

\begin{theorem} $Kite_{p,2}$ graphs are determined by their adjacency spectrum for all $p$.
\end{theorem}

\begin{proof}

If $p=1$ or $p=2$, $Kite_{p,2}$ graphs are actually the path graphs $P_{3}$ or $P_{4}$. Also if $p=3$, then we obtain the lollipop graph $H_{5,3}$. As is known, these graphs are already \textit{DAS} \cite{3}. Hence we will continue our proof for $p\geq4$. For a given graph $G$ with $n$ vertices and $m$ edges, assume that $G$ is cospectral with $Kite_{p,2}$. Then by Lemma 2.3 and Lemma 2.4, $n=p+2$, $m=\left(
                                       \begin{array}{c}
                                         p \\
                                         2 \\
                                       \end{array}
                                     \right)+2 = \frac{p^{2}-p+4}{2}$ and $t(G)=t(Kite_{p,2})=\left(
                                       \begin{array}{c}
                                         p \\
                                         3 \\
                                       \end{array}
                                     \right)=\frac{p^{3}-3p^{2}+2p}{6}$. From Lemma 3.2.1, $w(G)\geq p-2q+1$. When $q=2$,  $w(G)\geq p-3=n-5$. It's well-known that complete graph $K_{n}$ is \textit{DS}. So $w(G)\neq n$. If $w(G)= n-1=p+1$, then $G$ contains at least one clique with size $p-1$. It means that the edge number of $G$ is greater than or equal to $\left(
                                       \begin{array}{c}
                                         p+1 \\
                                         2 \\
                                       \end{array}
                                     \right)$. But it is a contradiction since $\left(
                                       \begin{array}{c}
                                         p+1 \\
                                         2 \\
                                       \end{array}
                                     \right)>\left(
                                       \begin{array}{c}
                                         p \\
                                         2 \\
                                       \end{array}
                                     \right)+2=m$. Hence, $w(G)\neq n-1$. Because of these, $n-5\leq w(G)\leq n-2$. Let us investigate the three cases, respectively, $w(G)= n-5$, $w(G)= n-4$, $w(G)= n-3$.

\textbf{CASE 1 :} Let $w(G)= n-5$. Then $w(G)= p-3$. So, $G$ contains at least one clique with size $p-3$. This clique is denoted by $K_{p-3}$. Let us label the five vertices, respectively, with $1,2,3,4,5$ which are not in the clique $K_{p-3}$ and call the set of these five vertices with $A=\{1,2,3,4,5\}$. We demonstrate this case by the following figure.

\begin{figure}[htbp] {}
\centering
\includegraphics[width=2.3cm,angle=0,height=2.3cm]{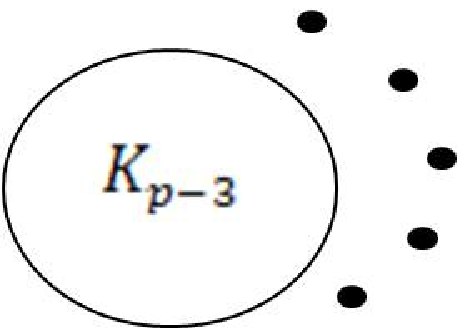}
\caption{}
\end{figure}

For $i\in A$, $x_{i}$ denotes the number of adjacent vertices of $i$ in $K_{p-3}$. By the fact that $w(G)=p-3$, for all $i\in A$ we say \begin{equation}x_{i}\leq p-4\end{equation} Also, $x_{i\wedge j}$ denotes the number of common adjacent vertices in $K_{p-3}$ of $i$ and $j$ such that $i,j\in A$ and $i<j$. Similarly, if $i\sim j$ then \begin{equation}x_{i\wedge j}\leq p-5\end{equation} Moreover, $d$ denotes the number of edges between the vertices of $A$ and $\alpha$ denotes the number of cliques with size 3 which are composed by vertices of $A$.

First of all, since the number of edges of $G$ is equal to $m$,

$m=\left(                                       \begin{array}{c}
                                         p \\
                                         2 \\
                                       \end{array}
                                     \right)+2 =\left(
                                       \begin{array}{c}
                                         p-3 \\
                                         2 \\
                                       \end{array}
                                     \right)+\sum_{i=1}^{5}x_{i}+d $. It follows that
                                     \begin{equation}\sum_{i=1}^{5}x_{i}+d= \left(                                       \begin{array}{c}
                                         p \\
                                         2 \\
                                       \end{array}
                                     \right)+2-\left(
                                       \begin{array}{c}
                                         p-3 \\
                                         2 \\
                                       \end{array}
                                     \right)=3p-4\end{equation} Similarly, by using $t(G)=\left(                                       \begin{array}{c}
                                         p \\
                                         3 \\
                                       \end{array}
                                     \right)$, we get

                                     $\left(                                       \begin{array}{c}
                                         p \\
                                         3 \\
                                       \end{array}
                                     \right)=\left(                                       \begin{array}{c}
                                         p-3 \\
                                         3 \\
                                       \end{array}
                                     \right)+\sum_{i=1}^{5}\left(                                       \begin{array}{c}
                                         x_{i} \\
                                         2 \\
                                       \end{array}
                                     \right)+\sum_{i\sim j}x_{i\wedge j}+\alpha$. Hence, we have

                                     \begin{equation}\sum_{i=1}^{5}\left(                                       \begin{array}{c}
                                         x_{i} \\
                                         2 \\
                                       \end{array}
                                     \right)+\sum_{i\sim j}x_{i\wedge j}+\alpha= \left(                                       \begin{array}{c}
                                         p \\
                                         3 \\
                                       \end{array}
                                     \right)-\left(                                       \begin{array}{c}
                                         p-3 \\
                                         3 \\
                                       \end{array}
                                     \right)=\frac{3p^{2}}{2}-\frac{15p}{2}+10\end{equation}

If $p=4$, then $w(G)=n-5=p-3=1$. Clearly, this is contradiction. Also if $p=5$, then $w(G)=n-5=p-3=2$ which implies $t(G)=0$. Again this is a contradiction. For this reason, we will continue for $p\geq6$.

Clearly, $0\leq d\leq10$. So, we will investigate the cases of $d$.
\vspace{3mm}

\textit{Subcase 1}

Let $d=0$. Then, $\sum_{i\sim j}x_{i\wedge j}+\alpha=0$ and from (3), we have \begin{equation}\sum_{i=1}^{5}x_{i}=3p-4\end{equation} Hence, we get

\begin{equation*}\sum_{i=1}^{5}\left(                                       \begin{array}{c}
                                         x_{i} \\
                                         2 \\
                                       \end{array}
                                     \right)+\sum_{i\sim j}x_{i\wedge j}+\alpha=\sum_{i=1}^{5}\left(                                       \begin{array}{c}
                                         x_{i} \\
                                         2 \\
                                       \end{array}
                                     \right)\end{equation*} Clearly, \begin{equation*}\sum_{i=1}^{5}\left(                                       \begin{array}{c}
                                         x_{i} \\
                                         2 \\
                                       \end{array}
                                     \right)\leq max\{\sum_{i=1}^{5}\left(                                       \begin{array}{c}
                                         x_{i} \\
                                         2 \\
                                       \end{array}
                                     \right)\}\end{equation*} Since, the spectrum of $G$ does not contain zero, $G$ has not an isolated vertex. So, from this fact and (1), we get $1\leq x_{i}\leq p-4$ for all $i\in A$. Hence, by (5), we get

                                     \begin{eqnarray} \sum_{i=1}^{5}\left(                                       \begin{array}{c}
                                         x_{i} \\
                                         2 \\
                                       \end{array}
                                     \right) &\leq&  max\{\sum_{i=1}^{5}\left(                                       \begin{array}{c}
                                         x_{i} \\
                                         2 \\
                                       \end{array}
                                     \right)\} \notag \\ &\leq& 3\left(                                       \begin{array}{c}
                                         p-4 \\
                                         2 \\
                                       \end{array}
                                     \right) + \left(                                       \begin{array}{c}
                                         7 \\
                                         2 \\
                                       \end{array}
                                     \right)\notag \\ &=& \frac{3p^{2}}{2}-\frac{27p}{2}+51\end{eqnarray}

From (1) and (5), $3p-4\leq5(p-4)$ which implies $8\leq p$. Where $8\leq p$, \begin{equation}\frac{3p^{2}}{2}-\frac{27p}{2}+51<\frac{3p^{2}}{2}-\frac{15p}{2}+10 \end{equation} This means that, $\sum_{i=1}^{5}\left(                                       \begin{array}{c}
                                         x_{i} \\
                                         2 \\
                                       \end{array}
                                     \right)<\frac{3p^{2}}{2}-\frac{15p}{2}+10$. But this result contradicts with (4).

\textit{Subcase 2}

Let $d=1$. Then $\alpha=0$ and from (3) we get \begin{equation}\sum_{i=1}^{5}x_{i}=3p-5\end{equation} Since $d=1$ and by (2), $\sum_{i\sim j}x_{i\wedge j}\leq p-5$. From here and (1), we have \begin{eqnarray} \sum_{i=1}^{5}\left(                                       \begin{array}{c}
                                         x_{i} \\
                                         2 \\
                                       \end{array}
                                     \right)+\sum_{i\sim j}x_{i\wedge j}+\alpha &\leq& max\{\sum_{i=1}^{5}\left(                                       \begin{array}{c}
                                         x_{i} \\
                                         2 \\
                                       \end{array}
                                     \right)\}+p-5 \notag \\ &\leq& 3\left(                                       \begin{array}{c}
                                         p-4 \\
                                         2 \\
                                       \end{array}
                                     \right) + \left(                                       \begin{array}{c}
                                         7 \\
                                         2 \\
                                       \end{array}
                                     \right)+p-5 \notag \\
                                     &=& \frac{3p^{2}}{2}-\frac{25p}{2}+46\end{eqnarray}

By using (8) and (1), we obtain $3p-5\leq5(p-4)$ which implies $15\leq2p$. Since $p$ is an integer, $8\leq p$. Where  $8\leq p$, \begin{equation}\frac{3p^{2}}{2}-\frac{25p}{2}+46<\frac{3p^{2}}{2}-\frac{15p}{2}+10\end{equation}
This means that, $\sum_{i=1}^{5}\left(                                       \begin{array}{c}
                                         x_{i} \\
                                         2 \\
                                       \end{array}
                                     \right)+\sum_{i\sim j}x_{i\wedge j}+\alpha<\frac{3p^{2}}{2}-\frac{15p}{2}+10$. But this result contradicts with (4).

\textit{Subcase 3}

Let $d=2$. Then $\alpha=0$ and by (3), we get \begin{equation}\sum_{i=1}^{5}x_{i}=3p-6\end{equation} By using similar way with last subcase, we obtain

\begin{eqnarray} \sum_{i=1}^{5}\left(                                       \begin{array}{c}
                                         x_{i} \\
                                         2 \\
                                       \end{array}
                                     \right)+\sum_{i\sim j}x_{i\wedge j}+\alpha &\leq& 3\left(                                       \begin{array}{c}
                                         p-4 \\
                                         2 \\
                                       \end{array}
                                     \right) + \left(                                       \begin{array}{c}
                                         6 \\
                                         2 \\
                                       \end{array}
                                     \right)+2(p-5) \notag \\ &=&\frac{3p^{2}}{2}-\frac{23p}{2}+35\end{eqnarray}

and $7\leq p$. If  $7\leq p$, we have \begin{equation}\frac{3p^{2}}{2}-\frac{23p}{2}+35<\frac{3p^{2}}{2}-\frac{15p}{2}+10\end{equation}
By (12) and (13), we get $\sum_{i=1}^{5}\left(                                       \begin{array}{c}
                                         x_{i} \\
                                         2 \\
                                       \end{array}
                                     \right)+\sum_{i\sim j}x_{i\wedge j}+\alpha<\frac{3p^{2}}{2}-\frac{15p}{2}+10$. This result contradicts with (4) as in Subcase 2.

\textit{Subcase 4}

Let $d=3$. Then $max\{\alpha\}=1$ and \begin{equation*}\sum_{i=1}^{5}x_{i}=3p-7\end{equation*} By using similar way again, we obtain

\begin{eqnarray} \sum_{i=1}^{5}\left(                                       \begin{array}{c}
                                         x_{i} \\
                                         2 \\
                                       \end{array}
                                     \right)+\sum_{i\sim j}x_{i\wedge j}+\alpha &\leq& 3\left(                                       \begin{array}{c}
                                         p-4 \\
                                         2 \\
                                       \end{array}
                                     \right) + \left(                                       \begin{array}{c}
                                         5 \\
                                         2 \\
                                       \end{array}
                                     \right)+3(p-5)+1 \notag \\ &=& \frac{3p^{2}}{2}-\frac{21p}{2}+26\end{eqnarray}

Since  $p\geq6$, we have \begin{equation}\frac{3p^{2}}{2}-\frac{21p}{2}+26<\frac{3p^{2}}{2}-\frac{15p}{2}+10\end{equation}
By (14) and (15), we have $\sum_{i=1}^{5}\left(                                       \begin{array}{c}
                                         x_{i} \\
                                         2 \\
                                       \end{array}
                                     \right)+\sum_{i\sim j}x_{i\wedge j}+\alpha<\frac{3p^{2}}{2}-\frac{15p}{2}+10$. We get same contradiction with (4).

\textit{Subcase 5}

Let $d=4$.Then $max\{\alpha\}=1$ and $\sum_{i=1}^{5}x_{i}=3p-8$. Similarly, we obtain

\begin{eqnarray} \sum_{i=1}^{5}\left(                                       \begin{array}{c}
                                         x_{i} \\
                                         2 \\
                                       \end{array}
                                     \right)+\sum_{i\sim j}x_{i\wedge j}+\alpha &\leq& 3\left(                                       \begin{array}{c}
                                         p-4 \\
                                         2 \\
                                       \end{array}
                                     \right) + \left(                                       \begin{array}{c}
                                         4 \\
                                         2 \\
                                       \end{array}
                                     \right)+4(p-5)+1 \notag \\ &=& \frac{3p^{2}}{2}-\frac{19p}{2}+17\end{eqnarray}

Since  $p\geq6$, we have \begin{equation}\frac{3p^{2}}{2}-\frac{19p}{2}+17<\frac{3p^{2}}{2}-\frac{15p}{2}+10\end{equation}

By (16) and (17), we get same contradiction with (4).

\textit{Subcase 6}

Let $d=5$.Then $max\{\alpha\}=2$ and $\sum_{i=1}^{5}x_{i}=3p-9$. Similarly, we obtain

\begin{eqnarray} \sum_{i=1}^{5}\left(                                       \begin{array}{c}
                                         x_{i} \\
                                         2 \\
                                       \end{array}
                                     \right)+\sum_{i\sim j}x_{i\wedge j}+\alpha &\leq& 3\left(                                       \begin{array}{c}
                                         p-4 \\
                                         2 \\
                                       \end{array}
                                     \right) + \left(                                       \begin{array}{c}
                                         3 \\
                                         2 \\
                                       \end{array}
                                     \right)+5(p-5)+2 \notag \\ &=&\frac{3p^{2}}{2}-\frac{17p}{2}+10\end{eqnarray}

Since  $p\geq6$, we get \begin{equation}\frac{3p^{2}}{2}-\frac{17p}{2}+10<\frac{3p^{2}}{2}-\frac{15p}{2}+10\end{equation}

By (18) and (19), we get same contradiction with (4).

\textit{Subcase 7}

Let $d=6$.Then $max\{\alpha\}=4$ and $\sum_{i=1}^{5}x_{i}=3p-10$. Similarly, we obtain

\begin{eqnarray} \sum_{i=1}^{5}\left(                                       \begin{array}{c}
                                         x_{i} \\
                                         2 \\
                                       \end{array}
                                     \right)+\sum_{i\sim j}x_{i\wedge j}+\alpha &\leq& 3\left(                                       \begin{array}{c}
                                         p-4 \\
                                         2 \\
                                       \end{array}
                                     \right) + \left(                                       \begin{array}{c}
                                         2 \\
                                         2 \\
                                       \end{array}
                                     \right)+6(p-5)+4 \notag \\ &=&\frac{3p^{2}}{2}-\frac{15p}{2}+5\end{eqnarray}

Since  $p\geq6$, we have \begin{equation}\frac{3p^{2}}{2}-\frac{15p}{2}+5<\frac{3p^{2}}{2}-\frac{15p}{2}+10\end{equation}

By (20) and (21), we get same contradiction with (4).

\textit{Subcase 8}

Let $d=7$.Then $max\{\alpha\}=4$ and $\sum_{i=1}^{5}x_{i}=3p-11$. Also here,  \begin{equation*}\sum_{i\sim j}x_{i\wedge j}\leq  \sum_{i=1}^{5}x_{i}+2(p-5)=5p-21\end{equation*}

Hence, in the same way as former subcases, we obtain

\begin{eqnarray} \sum_{i=1}^{5}\left(                                       \begin{array}{c}
                                         x_{i} \\
                                         2 \\
                                       \end{array}
                                     \right)+\sum_{i\sim j}x_{i\wedge j}+\alpha &\leq& 3\left(                                       \begin{array}{c}
                                         p-4 \\
                                         2 \\
                                       \end{array}
                                     \right) + 5p-21+4 \notag \\ &=& \frac{3p^{2}}{2}-\frac{17p}{2}+13\end{eqnarray}

Since  $p\geq6$, we get \begin{equation}\frac{3p^{2}}{2}-\frac{17p}{2}+13<\frac{3p^{2}}{2}-\frac{15p}{2}+10\end{equation}

So, by (22) and (23), we get same contradiction with (4).

\textit{Subcase 9}

Let $d=8$. Then $max\{\alpha\}=5$ and $\sum_{i=1}^{5}x_{i}=3p-12$. Such as in the last subcase, we get \begin{equation*}\sum_{i\sim j}x_{i\wedge j}\leq  \sum_{i=1}^{5}x_{i}+3(p-5)=6p-27\end{equation*}

Hence, we obtain

\begin{eqnarray} \sum_{i=1}^{5}\left(                                       \begin{array}{c}
                                         x_{i} \\
                                         2 \\
                                       \end{array}
                                     \right)+\sum_{i\sim j}x_{i\wedge j}+\alpha &\leq& 3\left(                                       \begin{array}{c}
                                         p-4 \\
                                         2 \\
                                       \end{array}
                                     \right) + 6p-27+5 \notag \\ &=& \frac{3p^{2}}{2}-\frac{15p}{2}+8\end{eqnarray}

Since  $p\geq6$, we get \begin{equation}\frac{3p^{2}}{2}-\frac{15p}{2}+8<\frac{3p^{2}}{2}-\frac{15p}{2}+10\end{equation}

So, by (24) and (25), we get same contradiction with (4).

\textit{Subcase 10}

Let $d=9$. Then $max\{\alpha\}=7$ and $\sum_{i=1}^{5}x_{i}=3p-13$. Similarly, we get \begin{equation*}\sum_{i\sim j}x_{i\wedge j}\leq  \sum_{i=1}^{5}x_{i}+4(p-5)=7p-33\end{equation*}
Hence, we obtain

\begin{eqnarray} \sum_{i=1}^{5}\left(                                       \begin{array}{c}
                                         x_{i} \\
                                         2 \\
                                       \end{array}
                                     \right)+\sum_{i\sim j}x_{i\wedge j}+\alpha &\leq& 2\left(                                       \begin{array}{c}
                                         p-4 \\
                                         2 \\
                                       \end{array}
                                     \right) + \left(                                       \begin{array}{c}
                                         p-5 \\
                                         2 \\
                                       \end{array}
                                     \right)+ 7p-33+7 \notag \\ &=& \frac{3p^{2}}{2}-\frac{15p}{2}+9\end{eqnarray}

Clearly, if  $p\geq6$, then \begin{equation}\frac{3p^{2}}{2}-\frac{15p}{2}+9<\frac{3p^{2}}{2}-\frac{15p}{2}+10\end{equation}

By (26) and (27), we get same contradiction with (4).

\textit{Subcase 11}

Let $d=10$. Then $max\{\alpha\}=10$ and we get \begin{equation*}\sum_{i=1}^{5}x_{i}=3p-14\end{equation*} Also, we have \begin{equation*}\sum_{i\sim j}x_{i\wedge j}\leq  2(\sum_{i=1}^{5}x_{i})=6p-28\end{equation*}  Thus, we obtain

\begin{eqnarray} \sum_{i=1}^{5}\left(                                       \begin{array}{c}
                                         x_{i} \\
                                         2 \\
                                       \end{array}
                                     \right)+\sum_{i\sim j}x_{i\wedge j}+\alpha &\leq& 2\left(                                       \begin{array}{c}
                                         p-4 \\
                                         2 \\
                                       \end{array}
                                     \right) + \left(                                       \begin{array}{c}
                                         p-6 \\
                                         2 \\
                                       \end{array}
                                     \right)+ 6p-28+10 \notag \\ &=& \frac{3p^{2}}{2}-\frac{19p}{2}+23\end{eqnarray}

If $p=6$, then $\sum_{i=1}^{5}x_{i}=4$. It follows that $\exists i\in A,x_{i}=0$ and so $\forall i,j \in A$ , $i\sim j$. By using the fact that $\exists i\in A$ , $x_{i}$=0, we get  \begin{equation*}\sum_{i\sim j}x_{i\wedge j} \leq  6(p-5)=6\end{equation*} and \begin{eqnarray*}\sum_{i=1}^{5}\left(                                       \begin{array}{c}
                                         x_{i} \\
                                         2 \\
                                       \end{array}
                                     \right)+\sum_{i\sim j}x_{i\wedge j}+\alpha &\leq& 2\left(                                       \begin{array}{c}
                                         2 \\
                                         2 \\
                                       \end{array}
                                     \right) +6+10 \notag \\ &=&18\end{eqnarray*} From (4), we get $\frac{3p^{2}}{2}-\frac{15p}{2}+10=19$. Thus,  \begin{equation}\sum_{i=1}^{5}\left(                                       \begin{array}{c}
                                         x_{i} \\
                                         2 \\
                                       \end{array}
                                     \right)+\sum_{i\sim j}x_{i\wedge j}+\alpha<\frac{3p^{2}}{2}-\frac{15p}{2}+10\end{equation}

If $p\geq7$, then \begin{equation}\frac{3p^{2}}{2}-\frac{19p}{2}+23<\frac{3p^{2}}{2}-\frac{15p}{2}+10\end{equation}

By (28),(29) and (30), we have contradiction with (4). \\

From Subcase 1 to Subcase 11,  $w(G)\neq n-5$.

\textbf{CASE 2:} Let $w(G)= n-4$. Then $w(G)= p-2$. So, $G$ contains at least one clique with size $p-2$.  We use similar notations with Case 1.

\begin{figure}[htbp] {}
\centering
\includegraphics[width=2.3cm,angle=0,height=2.3cm]{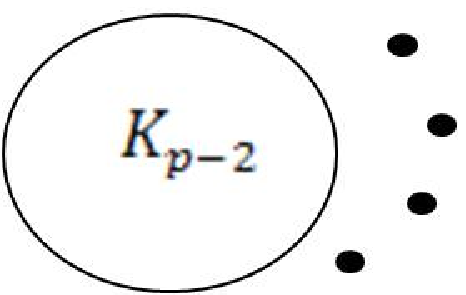}
\caption{}
\end{figure}

By the fact that $w(G)=p-2$, for all $i\in B$ we say $x_{i}\leq p-3$ such that $B=\{1,2,3,4\}$. Also,when $i\sim j$, $x_{i\wedge j}\leq p-5$ such that $i,j\in B$ and $i<j$. Since the number of edges of $G$ is equal to $m$, we get,

\begin{equation*}m=\left(                                       \begin{array}{c}
                                         p \\
                                         2 \\
                                       \end{array}
                                     \right)+2 =\left(
                                       \begin{array}{c}
                                         p-2 \\
                                         2 \\
                                       \end{array}
                                     \right)+\sum_{i=1}^{4}x_{i}+d \end{equation*} It follows that
                                     \begin{equation}\sum_{i=1}^{4}x_{i}+d= \left(                                       \begin{array}{c}
                                         p \\
                                         2 \\
                                       \end{array}
                                     \right)+2-\left(
                                       \begin{array}{c}
                                         p-2 \\
                                         2 \\
                                       \end{array}
                                     \right)=2p-1\end{equation} Also, from $t(G)=t(Kite_{p,2})$, we get

                                       \begin{equation*}\left(                                       \begin{array}{c}
                                         p \\
                                         3 \\
                                       \end{array}
                                     \right)=\left(                                       \begin{array}{c}
                                         p-2 \\
                                         3 \\
                                       \end{array}
                                     \right)+\sum_{i=1}^{4}\left(                                       \begin{array}{c}
                                         x_{i} \\
                                         2 \\
                                       \end{array}
                                     \right)+\sum_{i\sim j}x_{i\wedge j}+\alpha\end{equation*} Hence, we have

                                     \begin{eqnarray}\sum_{i=1}^{4}\left(                                       \begin{array}{c}
                                         x_{i} \\
                                         2 \\
                                       \end{array}
                                     \right)+\sum_{i\sim j}x_{i\wedge j}+\alpha &=& \left(                                       \begin{array}{c}
                                         p \\
                                         3 \\
                                       \end{array}
                                     \right)-\left(                                       \begin{array}{c}
                                         p-2 \\
                                         3 \\
                                       \end{array}
                                     \right) \notag \\ &=&(p-2)^{2} \notag \\ &=&p^{2}-4p+4\end{eqnarray}

If $p=4$, then $w(G)=n-4=p-2=2$. This means that, $t(G)=0$. But it contradicts with $t(G)=t(Kite_{4,2})=4$. So, we will continue to investigate for $p\geq5$. Obviously, in this case $0\leq d\leq6$.

\textit{Subcase 1}

Let $d=0$. Then, $\sum_{i\sim j}x_{i\wedge j}+\alpha=0$ and \begin{equation}\sum_{i=1}^{4}x_{i}=2p-1\end{equation}

Clearly, \begin{equation}\sum_{i=1}^{4}\left(                                       \begin{array}{c}
                                         x_{i} \\
                                         2 \\
                                       \end{array}
                                     \right)\leq max\{\sum_{i=1}^{5}\left(                                       \begin{array}{c}
                                         x_{i} \\
                                         2 \\
                                       \end{array}
                                     \right)\}\end{equation}

Since $G$ does not contain an isolated vertex, $1\leq x_{i}\leq p-3$ for all $i\in B$. Hence, by (33) and(34), we get

                                     \begin{eqnarray} \sum_{i=1}^{4}\left(                                       \begin{array}{c}
                                         x_{i} \\
                                         2 \\
                                       \end{array}
                                     \right)+\sum_{i\sim j}x_{i\wedge j}+\alpha &\leq&  max\{\sum_{i=1}^{4}\left(                                       \begin{array}{c}
                                         x_{i} \\
                                         2 \\
                                       \end{array}
                                     \right)\} \notag \\ &\leq& 2\left(                                       \begin{array}{c}
                                         p-3 \\
                                         2 \\
                                       \end{array}
                                     \right) + \left(                                       \begin{array}{c}
                                         4 \\
                                         2 \\
                                       \end{array}
                                     \right) \notag \\ &=&p^{2}-7p+18\end{eqnarray}

Clearly, for $p\geq5$,  \begin{equation}p^{2}-7p+18<p^{2}-4p+4\end{equation} By (35) and (36), we get \begin{equation*}\sum_{i=1}^{4}\left(                                       \begin{array}{c}
                                         x_{i} \\
                                         2 \\
                                       \end{array}
                                     \right)+\sum_{i\sim j}x_{i\wedge j}+\alpha<p^{2}-4p+4\end{equation*} But this result contradicts with (32).

\textit{Subcase 2}

Let $d=1$. Then, $\alpha=0$ and $\sum_{i=1}^{4}x_{i}=2p-2$. If $p=5$, then $\sum_{i=1}^{4}x_{i}=8$. So for all $i\in B$, $x_{i}=2$. Since $d=1$, we get $\sum_{i\sim j}x_{i\wedge j}=1$. Hence, $\sum_{i=1}^{4}\left(                                       \begin{array}{c}
                                         x_{i} \\
                                         2 \\
                                       \end{array}
                                     \right)+\sum_{i\sim j}x_{i\wedge j}+\alpha=5$ but from (33) $\sum_{i=1}^{4}\left(                                       \begin{array}{c}
                                         x_{i} \\
                                         2 \\
                                       \end{array}
                                     \right)+\sum_{i\sim j}x_{i\wedge j}+\alpha=9$. Because of this contradiction, $p\neq5$.

Also, we obtain             \begin{eqnarray} \sum_{i=1}^{4}\left(                                       \begin{array}{c}
                                         x_{i} \\
                                         2 \\
                                       \end{array}
                                     \right)+\sum_{i\sim j}x_{i\wedge j}+\alpha &\leq&  max\{\sum_{i=1}^{4}\left(                                       \begin{array}{c}
                                         x_{i} \\
                                         2 \\
                                       \end{array}
                                     \right)\}+p-4 \notag \\ &\leq& 2\left(                                       \begin{array}{c}
                                         p-3 \\
                                         2 \\
                                       \end{array}
                                     \right) + \left(                                       \begin{array}{c}
                                         4 \\
                                         2 \\
                                       \end{array}
                                     \right)+p-4 \notag \\ &=& p^{2}-6p+14\end{eqnarray}

If $p\geq6$, then  \begin{equation}p^{2}-6p+14<p^{2}-4p+4\end{equation} By (37) and (38), we contradict with (32).

\textit{Subcase 3}

Let $d=2$. Then, $\alpha=0$ and $\sum_{i=1}^{4}x_{i}=2p-3$. Also,  $\sum_{i\sim j}x_{i\wedge j}\leq2(p-4)$. Hence, as in last subcase, we obtain

\begin{eqnarray} \sum_{i=1}^{4}\left(                                       \begin{array}{c}
                                         x_{i} \\
                                         2 \\
                                       \end{array}
                                     \right)+\sum_{i\sim j}x_{i\wedge j}+\alpha &\leq& 2\left(                                       \begin{array}{c}
                                         p-3 \\
                                         2 \\
                                       \end{array}
                                     \right) + \left(                                       \begin{array}{c}
                                         3 \\
                                         2 \\
                                       \end{array}
                                     \right)+2p-8 \notag \\ &=& p^{2}-5p+7\end{eqnarray}

If $p\geq5$, then \begin{equation}p^{2}-5p+7<p^{2}-4p+4\end{equation} By (39) and (40), we contradict with (32).

\textit{Subcase 4}

Let $d=3$. Then, $max\{\alpha\}=1$, $\sum_{i=1}^{4}x_{i}=2p-4$ and $\sum_{i\sim j}x_{i\wedge j}\leq3(p-4)$. So, we get

\begin{eqnarray} \sum_{i=1}^{4}\left(                                       \begin{array}{c}
                                         x_{i} \\
                                         2 \\
                                       \end{array}
                                     \right)+\sum_{i\sim j}x_{i\wedge j}+\alpha &\leq& 2\left(                                       \begin{array}{c}
                                         p-3 \\
                                         2 \\
                                       \end{array}
                                     \right) + \left(                                       \begin{array}{c}
                                         2 \\
                                         2 \\
                                       \end{array}
                                     \right)+3p-12+1 \notag \\ &=& p^{2}-4p+2\end{eqnarray}

For $p\geq5$,  \begin{equation}p^{2}-4p+2<p^{2}-4p+4\end{equation} Again, we contradict with (32).

\textit{Subcase 5}

Let $d=4$. Then, $max\{\alpha\}=1$, $\sum_{i=1}^{4}x_{i}=2p-5$ and $\sum_{i\sim j}x_{i\wedge j}\leq \sum_{i=1}^{4}x_{i}=2p-5$. Hence, we get

\begin{eqnarray} \sum_{i=1}^{4}\left(                                       \begin{array}{c}
                                         x_{i} \\
                                         2 \\
                                       \end{array}
                                     \right)+\sum_{i\sim j}x_{i\wedge j}+\alpha &\leq& 2\left(                                       \begin{array}{c}
                                         p-3 \\
                                         2 \\
                                       \end{array}
                                     \right) + 2p-5+1 \notag \\ &=&p^{2}-5p+8\end{eqnarray}

For $p\geq5$,  \begin{equation}p^{2}-5p+8<p^{2}-4p+4\end{equation} Again, we contradict with (32).

\textit{Subcase 6}

Let $d=5$. Then, $max\{\alpha\}=2$ and $\sum_{i=1}^{4}x_{i}=2p-6$. Since $x_{i}\leq p-3$ and $\sum_{i=1}^{4}x_{i}=2p-6$, at most for one pair of adjacent vertices of $B$, $x_{i\wedge j}$ could be equal to $p-4$. Except of this vertex pair, $x_{i\wedge j}<p-4$. So, $\sum_{i\sim j}x_{i\wedge j}\leq \sum_{i=1}^{4}x_{i}+p-5$. Hence, we have

\begin{eqnarray} \sum_{i=1}^{4}\left(                                       \begin{array}{c}
                                         x_{i} \\
                                         2 \\
                                       \end{array}
                                     \right)+\sum_{i\sim j}x_{i\wedge j}+\alpha &\leq& 2\left(                                       \begin{array}{c}
                                         p-3 \\
                                         2 \\
                                       \end{array}
                                     \right) + 2p-6+p-5+2 \notag \\ &=& p^{2}-4p+3\end{eqnarray}

Clearly, for $p\geq5$,  \begin{equation}p^{2}-4p+3<p^{2}-4p+4\end{equation} By (45) and (46), we contradict with (32).

\textit{Subcase 7}

Let $d=6$. Then, $max\{\alpha\}=4$ and $\sum_{i=1}^{4}x_{i}=2p-7$. Same as last subcase, we get \begin{equation*}\sum_{i\sim j}x_{i\wedge j}\leq \sum_{i=1}^{4}x_{i}+2(p-5)=4p-17\end{equation*} Hence, we obtain

\begin{eqnarray} \sum_{i=1}^{4}\left(                                       \begin{array}{c}
                                         x_{i} \\
                                         2 \\
                                       \end{array}
                                     \right)+\sum_{i\sim j}x_{i\wedge j}+\alpha\leq \left(                                       \begin{array}{c}
                                         p-3 \\
                                         2 \\
                                       \end{array}
                                     \right) + \left(                                       \begin{array}{c}
                                         p-4 \\
                                         2 \\
                                       \end{array}
                                     \right)+4p-17+4=p^{2}-4p+3\end{eqnarray}
While $p\geq5$, we get  \begin{equation}p^{2}-4p+3<p^{2}-4p+4\end{equation} By (47) and (48), we contradict with (32).

Thus we have seen the same result with Case 1, that is $t(G)<t(Kite_{p,2})$. So, that is the same contradiction. Consequently, $w(G)\neq n-4$.

\textbf{CASE 3:} Let $w(G)= n-3= p-1$. So, $G$ contains at least one clique with size $p-1$.  We use similar notations with Case 1 and Case 2.

\begin{figure}[htbp] {}
\centering
\includegraphics[width=2.3cm,angle=0,height=2.3cm]{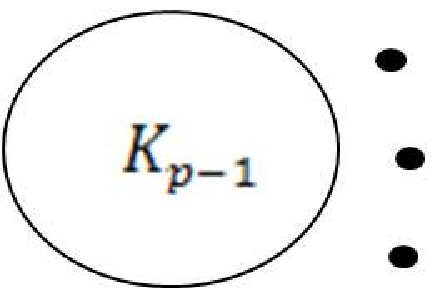}
\caption{}
\end{figure}

Since $w(G)=p-1$, for all $i\in C$, $x_{i}\leq p-2$ such that $C=\{1,2,3\}$. Also if $i\sim j$, then $x_{i\wedge j}\leq p-3$ such that $i,j\in C$ and $i<j$. By using the facts that edge number of $G$ is equal to $m$ and $t(G)=t(Kite_{p,2})$, we get the following equations,

\begin{equation}\sum_{i=1}^{3}x_{i}+d=m-\left(                                       \begin{array}{c}
                                         p-1 \\
                                         2 \\
                                       \end{array}
                                     \right)= \left(                                       \begin{array}{c}
                                         p \\
                                         2 \\
                                       \end{array}
                                     \right)+2-\left(
                                       \begin{array}{c}
                                         p-1 \\
                                         2 \\
                                       \end{array}
                                     \right)=p+1\end{equation}

                                     \begin{eqnarray}\sum_{i=1}^{3}\left(                                       \begin{array}{c}
                                         x_{i} \\
                                         2 \\
                                       \end{array}
                                     \right)+\sum_{i\sim j}x_{i\wedge j}+\alpha &=& t(Kite_{p,2})-\left(
                                       \begin{array}{c}
                                         p-1 \\
                                         3 \\
                                       \end{array}
                                     \right) \notag \\ &=& \left(                                       \begin{array}{c}
                                         p \\
                                         3 \\
                                       \end{array}
                                     \right)-\left(                                       \begin{array}{c}
                                         p-1 \\
                                         3 \\
                                       \end{array}
                                     \right) \notag \\ &=& \frac{p^{2}-3p+2}{2}\end{eqnarray}

In this case $0\leq d\leq3$.

\textit{Subcase 1}

Let $d=0$. Then, $\sum_{i\sim j}x_{i\wedge j}+\alpha=0$ and $\sum_{i=1}^{3}x_{i}=p+1$. So, we get \begin{equation*}\sum_{i=1}^{3}\left(                                       \begin{array}{c}
                                         x_{i} \\
                                         2 \\
                                       \end{array}
                                     \right)+\sum_{i\sim j}x_{i\wedge j}+\alpha=\sum_{i=1}^{3}\left(                                       \begin{array}{c}
                                         x_{i} \\
                                         2 \\
                                       \end{array}
                                     \right)\end{equation*} Since $G$ does not contain an isolated vertex, $x_{i}>0$ for all $i\in C$.Thus, we have

                                     \begin{eqnarray*} \sum_{i=1}^{3}\left(                                       \begin{array}{c}
                                         x_{i} \\
                                         2 \\
                                       \end{array}
                                     \right) &=& \left(                                       \begin{array}{c}
                                         x_{1} \\
                                         2 \\
                                       \end{array}
                                     \right)+\left(                                       \begin{array}{c}
                                         x_{2} \\
                                         2 \\
                                       \end{array}
                                     \right)+\left(                                       \begin{array}{c}
                                         x_{3} \\
                                         2 \\
                                       \end{array}
                                     \right)\notag \\ &<& \left(                                       \begin{array}{c}
                                         x_{1}+x_{2}+x_{3}-2 \\
                                         2 \\
                                       \end{array}
                                     \right) \notag \\ &=& \left(                                       \begin{array}{c}
                                         p-1 \\
                                         2 \\
                                       \end{array}
                                     \right)\notag \\ &=& \frac{p^{2}-3p+2}{2}\end{eqnarray*}

But this result contradicts with (50).
\textit{Subcase 2}

Let $d=1$. Then, $\alpha=0$ and $\sum_{i=1}^{3}x_{i}=p$. We may call the adjacent vertices in $C$ with $1$ and $2$. So, we get \begin{equation*}\sum_{i=1}^{3}\left(                                       \begin{array}{c}
                                         x_{i} \\
                                         2 \\
                                       \end{array}
                                     \right)+\sum_{i\sim j}x_{i\wedge j}+\alpha=\sum_{i=1}^{3}\left(                                       \begin{array}{c}
                                         x_{i} \\
                                         2 \\
                                       \end{array}
                                     \right)+x_{1\wedge2}\end{equation*}

Since $G$ does not contain any isolated vertex, $x_{3}>0$. If $x_{1}=0$ (or $x_{2}=0$), then $x_{2}+x_{3}=p$ and $\sum_{i=1}^{3}\left(                                       \begin{array}{c}
                                         x_{i} \\
                                         2 \\
                                       \end{array}
                                     \right)=\left(                                       \begin{array}{c}
                                         x_{2} \\
                                         2 \\
                                       \end{array}
                                     \right)+\left(                                       \begin{array}{c}
                                         x_{3} \\
                                         2 \\
                                       \end{array}
                                     \right)$. Since $p\geq4$ and $\forall i\in C$ $x_{i}\leq p-2$,

                                     \begin{eqnarray} \sum_{i=1}^{3}\left(                                       \begin{array}{c}
                                         x_{i} \\
                                         2 \\
                                       \end{array}
                                     \right)+x_{1\wedge2}&=&\left(                                       \begin{array}{c}
                                         x_{2} \\
                                         2 \\
                                       \end{array}
                                     \right)+\left(                                       \begin{array}{c}
                                         x_{3} \\
                                         2 \\
                                       \end{array}
                                     \right) \notag \\  &\leq& max{\left(                                       \begin{array}{c}
                                         x_{2} \\
                                         2 \\
                                       \end{array}
                                     \right)+\left(                                       \begin{array}{c}
                                         x_{3} \\
                                         2 \\
                                       \end{array}
                                     \right)} \notag \\ &\leq& \left(                                       \begin{array}{c}
                                         p-2 \\
                                         2 \\
                                       \end{array}
                                     \right)+\left(                                       \begin{array}{c}
                                         2 \\
                                         2 \\
                                       \end{array}
                                     \right) \notag \\ &=&\frac{p^{2}-5p+8}{2}<\frac{p^{2}-3p+2}{2}\end{eqnarray}

 If $x_{1}$,$x_{2}>0$, then by using  $x_{i}\leq p-2$ and $x_{i\wedge j}\leq p-3$ such that $i\sim j$,

 \begin{eqnarray} \sum_{i=1}^{3}\left(                                       \begin{array}{c}
                                         x_{i} \\
                                         2 \\
                                       \end{array}
                                     \right)+x_{1\wedge2} &\leq& max\{\sum_{i=1}^{3}\left(                                       \begin{array}{c}
                                         x_{i} \\
                                         2 \\
                                       \end{array}
                                     \right)\}+p-3 \notag \\ &\leq& \left(                                       \begin{array}{c}
                                         p-2 \\
                                         2 \\
                                       \end{array}
                                     \right)+p-3 \notag \\ &=&\frac{p^{2}-3p}{2}<\frac{p^{2}-3p+2}{2}\end{eqnarray}

By (51) and (52), we have contradiction with (50).

\textit{Subcase 3}

Let $d=2$. Then, $\alpha=0$ and $\sum_{i=1}^{3}x_{i}=p-1$. We may call the pair of adjacent vertices in $C$,respectively, with (1,2) and (2,3). Hence, we get \begin{equation}\sum_{i=1}^{3}\left(                                       \begin{array}{c}
                                         x_{i} \\
                                         2 \\
                                       \end{array}
                                     \right)+\sum_{i\sim j}x_{i\wedge j}+\alpha=\sum_{i=1}^{3}\left(                                       \begin{array}{c}
                                         x_{i} \\
                                         2 \\
                                       \end{array}
                                     \right)+x_{1\wedge2}+x_{2\wedge3}\end{equation}

If $x_{1}=0$ (or $x_{3}=0$), then $x_{2}+x_{3}=p-1$ and \begin{equation*}\sum_{i=1}^{3}\left(                                       \begin{array}{c}
                                         x_{i} \\
                                         2 \\
                                       \end{array}
                                     \right)+x_{1\wedge2}+x_{2\wedge3}=\left(                                       \begin{array}{c}
                                         x_{2} \\
                                         2 \\
                                       \end{array}
                                     \right)+\left(                                       \begin{array}{c}
                                         x_{3} \\
                                         2 \\
                                       \end{array}
                                     \right)+x_{2\wedge3}\end{equation*} Since $x_{i}\leq p-2$ and $x_{i\wedge j}\leq p-3$, we get

                                     \begin{eqnarray}\left(                                       \begin{array}{c}
                                         x_{2} \\
                                         2 \\
                                       \end{array}
                                     \right)+\left(                                       \begin{array}{c}
                                         x_{3} \\
                                         2 \\
                                       \end{array}
                                     \right)+x_{2\wedge3} &\leq& max\{\left(                                       \begin{array}{c}
                                         x_{2} \\
                                         2 \\
                                       \end{array}
                                     \right)+\left(                                       \begin{array}{c}
                                         x_{3} \\
                                         2 \\
                                       \end{array}
                                     \right)\}+p-3 \notag \\ &\leq& \left(                                       \begin{array}{c}
                                         p-2 \\
                                         2 \\
                                       \end{array}
                                     \right)+p-3 \notag \\ &=& \frac{p^{2}-3p}{2}<\frac{p^{2}-3p+2}{2}\end{eqnarray}

 If $x_{2}=0$, then $x_{1}+x_{3}=p-1$ and

  \begin{eqnarray}\left(                                       \begin{array}{c}
                                         x_{2} \\
                                         2 \\
                                       \end{array}
                                     \right)+\left(                                       \begin{array}{c}
                                         x_{3} \\
                                         2 \\
                                       \end{array}
                                     \right)+x_{2\wedge3} &\leq& max\{\left(                                       \begin{array}{c}
                                         x_{2} \\
                                         2 \\
                                       \end{array}
                                     \right)+\left(                                       \begin{array}{c}
                                         x_{3} \\
                                         2 \\
                                       \end{array}
                                     \right)\}\notag \\ &\leq& \left(                                       \begin{array}{c}
                                         p-2 \\
                                         2 \\
                                       \end{array}
                                     \right)\notag \\ &<&\left(                                       \begin{array}{c}
                                         p-1 \\
                                         2 \\
                                       \end{array}
                                     \right)\notag \\ &=&\frac{p^{2}-3p+2}{2}\end{eqnarray}

 If $x_{i}>0$ for all $i$, then,

 \begin{eqnarray} \sum_{i=1}^{3}\left(                                       \begin{array}{c}
                                         x_{i} \\
                                         2 \\
                                       \end{array}
                                     \right)+x_{1\wedge2}+x_{2\wedge3} &\leq& \sum_{i=1}^{3}\left(                                       \begin{array}{c}
                                         x_{i} \\
                                         2 \\
                                       \end{array}
                                     \right)+x_{1}+x_{2} \notag \\ &<& \left(                                       \begin{array}{c}
                                         x_{1}+x_{2}+x_{3} \\
                                         2 \\
                                       \end{array}
                                     \right)\notag \\ &=&\left(                                       \begin{array}{c}
                                        p-1 \\
                                         2 \\
                                       \end{array}
                                     \right)\notag \\ &=&\frac{p^{2}-3p+2}{2}\end{eqnarray}

By (53),(54) and (55), we have contradiction with (50).

\textit{Subcase 4}

Let $d=3$. Then, we have $\alpha=1$ and $\sum_{i=1}^{3}x_{i}=p-2$. Here, all of the vertices of $C$ are adjacent to each other. Hence, we get \begin{equation*}\sum_{i=1}^{3}\left(                                       \begin{array}{c}
                                         x_{i} \\
                                         2 \\
                                       \end{array}
                                     \right)+\sum_{i\sim j}x_{i\wedge j}+\alpha=\sum_{i=1}^{3}\left(                                       \begin{array}{c}
                                         x_{i} \\
                                         2 \\
                                       \end{array}
                                     \right)+x_{1\wedge2}+x_{2\wedge3}+x_{1\wedge3}+1\end{equation*}

Since $\sum_{i=1}^{3}x_{i}=p-2$, $\exists i \in C$, $x_{i}\neq0$. Without loss of generality, if  $x_{1}=x_{2}=0$ then $x_{3}=p-3$. Since $x_{i}\leq p-2$, we get

\begin{eqnarray}\sum_{i=1}^{3}\left(                                       \begin{array}{c}
                                         x_{i} \\
                                         2 \\
                                       \end{array}
                                     \right)+x_{1\wedge2}+x_{2\wedge3}+x_{1\wedge3}+1 &=& \left(                                       \begin{array}{c}
                                         x_{3} \\
                                         2 \\
                                       \end{array}
                                     \right)+1 \notag \\ &\leq&\left(                                       \begin{array}{c}
                                         p-2 \\
                                         2 \\
                                       \end{array}
                                     \right)+1\notag \\ &=&\frac{p^{2}-5p+8}{2}\notag \\ &<& \frac{p^{2}-3p+2}{2}\end{eqnarray}

Without loss of generality, if $x_{1}=0$ , then $x_{2}+x_{3}=p-2$ and \begin{equation*}\sum_{i=1}^{3}\left(                                       \begin{array}{c}
                                         x_{i} \\
                                         2 \\
                                       \end{array}
                                     \right)+x_{1\wedge2}+x_{2\wedge3}+x_{1\wedge3}+1=\left(                                       \begin{array}{c}
                                         x_{2} \\
                                         2 \\
                                       \end{array}
                                     \right)+\left(                                       \begin{array}{c}
                                         x_{3} \\
                                         2 \\
                                       \end{array}
                                     \right)+x_{2\wedge3}+1\end{equation*} Since $x_{i}\leq p-2$ and $x_{i\wedge j}\leq p-3$, we get

                                     \begin{eqnarray}\left(                                       \begin{array}{c}
                                         x_{2} \\
                                         2 \\
                                       \end{array}
                                     \right)+\left(                                       \begin{array}{c}
                                         x_{3} \\
                                         2 \\
                                       \end{array}
                                     \right)+x_{2\wedge3}+1 &\leq& max\{\left(                                       \begin{array}{c}
                                         x_{2} \\
                                         2 \\
                                       \end{array}
                                     \right)+\left(                                       \begin{array}{c}
                                         x_{3} \\
                                         2 \\
                                       \end{array}
                                     \right)\}+p-3+1 \notag \\ &\leq& \left(                                       \begin{array}{c}
                                         p-2 \\
                                         2 \\
                                       \end{array}
                                     \right)+p-3+1 \notag \\ &=& \frac{p^{2}-3p}{2}<\frac{p^{2}-3p+2}{2}\end{eqnarray}

 If $x_{i}>0$ for all $i$, then we get

 \begin{eqnarray} \sum_{i=1}^{3}\left(                                       \begin{array}{c}
                                         x_{i} \\
                                         2 \\
                                       \end{array}
                                     \right)+x_{1\wedge2}+x_{2\wedge3}+x_{1\wedge3}&\leq& \sum_{i=1}^{3}\left(                                       \begin{array}{c}
                                         x_{i} \\
                                         2 \\
                                       \end{array}
                                     \right)+\sum_{i=1}^{3} x_{i}+1 \notag \\ &<& \left(                                       \begin{array}{c}
                                         x_{1}+x_{2}+x_{3}+1 \\
                                         2 \\
                                       \end{array}
                                     \right) \notag \\ &=&\left(                                       \begin{array}{c}
                                        p-1 \\
                                         2 \\
                                       \end{array}
                                     \right)\notag \\ &=&\frac{p^{2}-3p+2}{2}\end{eqnarray}

By (56),(57) and (58), we have contradiction with (50).

Again, in this case, we have seen the same result $t(G)<t(Kite_{p,2})$ and got the same contradiction. Hence, we can write $w(G)\neq n-3$. From Case 1 to Case3, we can conclude that $w(G)=n-2=p$. So, $G$ must contain at least one clique with size $p$ and this is a maximum clique of $G$. So, there are two vertices out of a maximum clique of $G$. Let us label these two vertices with $1$ and $2$ and demonstrate this case in the following figure.

\begin{figure}[htbp] {}
\centering
\includegraphics[width=2.3cm,angle=0,height=2.3cm]{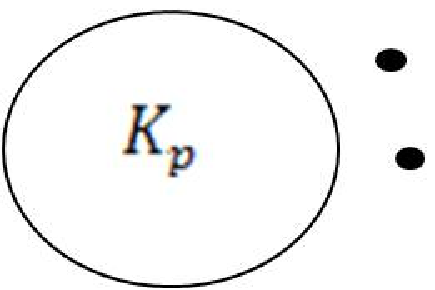}
\caption{}
\end{figure}

We denote the degrees of the vertices $1$ and $2$ respectively with $d_{1}$ and $d_{2}$. Then $d_{1}+d_{2}=2$. Since $G$ does not contain any isolated vertex, $d_{1}=d_{2}=1$. Thus, $G$ must be isomorphic to the one of the following three graphs.

\begin{figure}[htbp] {}
\centering
\includegraphics[width=2.3cm,angle=0,height=2.3cm]{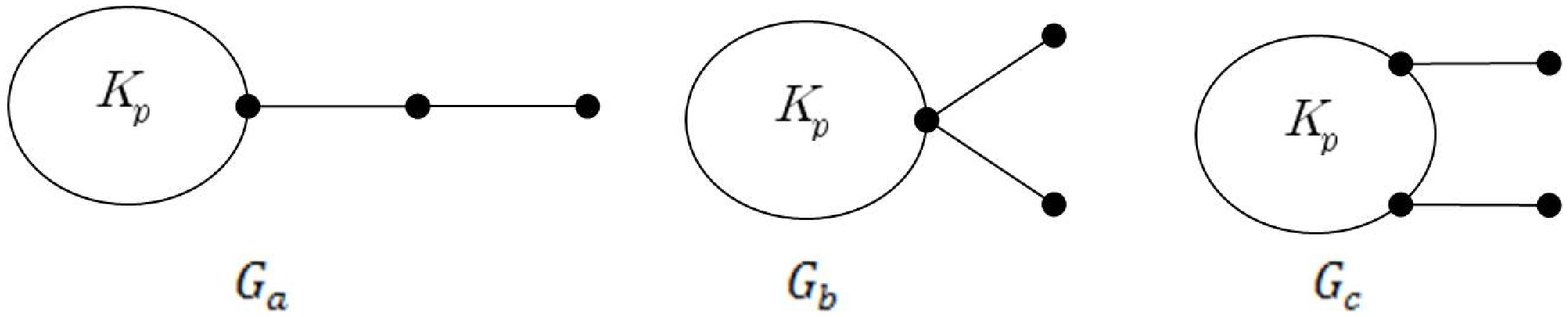}
\caption{}
\end{figure}

\bigskip

It is shown that $G_{b}$ is $DAS$ in \cite{4}. So, let us find the characteristic polynomial of the graph $G_{c}$. In this step, we use Lemma 2.1 again.

\begin{eqnarray*}
P_{A(G_{c})}(\lambda ) &=&\lambda P_{A(Kite_{p,1})}(\lambda
)-P_{A(Kite_{p-1,1})})(\lambda ) \\
&=&\lambda [(\lambda+1)^{p-2}(\lambda^{3}-(p-2)\lambda^{2}-\lambda p+p-2)] \\ && - [(\lambda+1)^{p-3}(\lambda^{3}-(p-3)\lambda^{2}-(p-1)\lambda+p-3)]\\
&=&(\lambda+1)^{p-3}[\lambda(\lambda^{3}-(p-2)\lambda^{2}-\lambda p+p-2)-(\lambda^{3}-(p-3)\lambda^{2}-(p-1)\lambda+p-3)]\\
&=&(\lambda+1)^{p-3}[\lambda^{4}+\lambda^{3}-3\lambda^{2}-3\lambda+3-p(1-2\lambda+\lambda^{3})]\\
&=&(\lambda+1)^{p-3}f(\lambda)
\end{eqnarray*}%

Hence, we can see that $G$ is not cospectral with $G_{c}$. So, $G$ is not isomorphic to $G_{c}$. Accordingly,   $G\cong G_{a}\cong Kite_{p,2}$
\end{proof}

In the final of the paper, we give some problems below.

\begin{conjecture} For $q>2$, Kite$_{p,q}$ graphs are DAS.
\end{conjecture}

\begin{problem} For a given simple and undirected graph G, let G be a DAS graph and contains a pendant vertex. Let H be the graph obtained from G by adding one edge to the pendant vertex of G. Then, is H  DAS?
\end{problem}

\begin{problem} Let G and H be graphs as in Problem 4.4. Which conditions must G satisfy to obtain the result that H is DAS?
\end{problem}


\end{document}